\documentclass[11pt,amssymb]{amsart}
\usepackage{amssymb}
\usepackage[mathscr]{eucal}

\newcommand{\Z}{{\mathbb Z}}
\newcommand{\C}{{\mathbb C}}

\newcommand{\N}{{\mathbb N}}
\newcommand{\I}{{\mathcal O}}

\newcommand{\G}{{\mathbb G}}

\newcommand{\ba}{\mbox{\boldmath{$\alpha$}}}
\newcommand{\bm}{\mbox{\boldmath{$\mu$}}}
\newtheorem{thm}{Theorem}[section]
\newtheorem{lemma}[thm]{Lemma}
\newtheorem{prop}[thm]{Proposition}

\usepackage[all,cmtip]{xy}
\begin{document}

\title[Representations of algebraic groups in characteristic $p$]{On abstract representations of the groups of rational points of algebraic groups in positive characteristic}

\author[M.~Boyarchenko]{Mitya Boyarchenko}

%\thanks{$^\flat$ Partially supported by NSF grant DMS-0502120  and the Humboldt Foundation.}

%\address{Department of Mathematics, University of Michigan,
%Ann Arbor, MI 48109}

\email{dmitriy.boyarchenko@gmail.com}

\author[I.A.~Rapinchuk]{Igor A. Rapinchuk}

\address{Department of Mathematics, Harvard University, Cambridge, MA 02138}

\email{rapinch@math.harvard.edu}

\begin{abstract}
We analyze the structure of a large class of connected algebraic rings over an algebraically closed field of positive characteristic using Greenberg's perfectization functor. We then give applications to rigidity problems for representations of Chevalley groups, recovering, in particular, a rigidity theorem of Seitz \cite{Sei}.
\end{abstract}

\maketitle

\section{Introduction}\label{S-1}

Our focus in this paper will be on two topics. First, we investigate the structure of a large class of connected algebraic rings over algebraically closed fields of positive characteristic. Second, we apply our structural results to obtain some new rigidity statements for abstract linear representations of elementary subgroups of Chevalley groups over commutative rings, building on the work of the second author in \cite{IR}.

Algebraic rings were first formally introduced by M.~Greenberg in \cite{G2} in connection with his analysis of reduction properties of schemes over discrete valuation rings.  Subsequently, he studied them systematically in \cite{G}, where a number of foundational results were established. Fix an algebraically closed field $K$, and denote by $Sch/K$ the category of all $K$-schemes and by
$Var/ K$ the full subcategory of reduced schemes of finite type over $K$.
%\footnotemark \footnotetext{We will frequently tacitly identify objects of the latter category with their sets of $K$-rational points.}
We define an {\it algebraic ring over $K$} to be a ring
%We will use the term {\it algebraic ring over $K$} to mean a ring
object $A$ in the category $Var/K$ whose underlying scheme is affine \footnotemark. In other words, $A$ is an affine algebraic variety over $K$ equipped with the structure of an associative ring that is defined by morphisms $\ba \colon A \times A \to A$ (``addition") and $\bm \colon A \times A \to A$ (``multiplication"). \footnotetext{This definition coincides with the one used by Greenberg \cite{G}, except that we do not require the underlying variety of $A$ to be irreducible. Furthermore, using Chevalley's structure theorem for algebraic groups, one shows that if $A$ is an irreducible ring variety over $K$, then $A$ is automatically affine (\cite[Proposition 4.3]{G} or \cite[Theorem 2.21]{IR}); thus, the affiness assumption in the definition can, strictly speaking, be omitted, but we will include it for clarity.} Note that for any finite-dimensional associative $K$-algebra $B$, the addition and multiplication operations endow the underlying $K$-vector space with the structure of an algebraic ring, which we will denote by $\tilde{B}.$ By construction, the set of $K$-rational points $\tilde{B}(K)$ is simply $B.$

%every finite-dimensional associative $K$-algebra $B$ can be viewed as an algebraic ring over $K$, which we will denote by $\tilde{B}$.
%if $B$ is a finite-dimensional $K$-algebra, then $B$ has a natural structure of an algebraic ring, which we will denote by $\tilde{B}.$

%A finite-dimensional $K$-algebra $B$ can obviously be viewed as an algebraic ring, which we will denote by $\tilde{B}.$ In more functorial terms, $\tilde{B}$ is the algebraic ring representing the functor from commutative $K$-algebras to rings defined by $R \mapsto B \otimes_K R.$ Clearly, we have $\tilde{B} (K) = B.$

The correspondence $B \mapsto \tilde{B}$ gives rise to a functor
$$
\mathcal{F} \colon \mathcal{C} \to \mathcal{D},
$$
where $\mathcal{C}$ is the category of associative finite-dimensional $K$-algebras and $\mathcal{D}$ is the category of connected algebraic rings over $K$.
If $K$ is a field of characteristic 0, then Greenberg showed that $\mathcal{F}$ is an equivalence of categories
(see  \cite[Proposition 5.1]{G} as well as \cite[Proposition 2.14]{IR}; for a topological proof in the case where $K = \C$, the reader can consult \cite[Lemma 5]{KS}). On the other hand, it is easy to see that $\mathcal{F}$ is no longer essentially surjective if ${\rm char}~K = p > 0$.
%then the correspondence this is no longer the case if ${\rm char}~K = p > 0.$

For example, let $A$ be the algebraic ring over $K$ whose underlying additive group is $\G_a \oplus \G_a$ with the componentwise operation and whose multiplication operation is given by
$$
(x_1, y_1) \cdot (x_2, y_2) = (x_1 x_2, x_1^p y_2 + x_2^p y_1).
$$
It is straightforward to check that $A$ is \emph{not} of the form $\tilde{B}$ for any finite-dimensional $K$-algebra $B$ (see \cite[Example 2.17]{IR}). Next, let $B = K[\varepsilon]$, with $\varepsilon^2 = 0$, be the $K$-algebra of dual numbers and $\tilde{B}$ the corresponding algebraic ring over $K$. Clearly, the underlying additive group of $\tilde{B}$ is again $\G_a \oplus \G_a$ and the multiplication operation is
$$
(x_1, y_1) \cdot (x_2, y_2) = (x_1 x_2, x_1 y_2 + x_2 y_1).
$$
Notice that the map
$$
(x, y ) \mapsto (x, y^p)
$$
defines a morphism of {\it algebraic} rings $\varphi \colon \tilde{B} \to A$ that induces an isomorphism of {\it abstract} rings $\varphi_K \colon \tilde{B}(K) = B \to A(K)$ on the $K$-rational points. Thus, we see that while $A$ is not isomorphic to an object in the image of $\mathcal{F}$, it is still closely related to a finite-dimensional $K$-algebra, up to an inseparable isogeny.
%up to an action of the Frobenius automorphism, our algebraic ring $A$ is very closely related to a finite-dimensional $K$-algebra.
Our first main result (and its proof) shows that the same holds true in the general case for connected algebraic rings $A$ over $K$ satisfying $p A = 0.$

\begin{thm}\label{T-1}
Let $A$ be a connected algebraic ring over an algebraically closed field $K$ of characteristic $p > 0$ such that $p A = 0$ and the ring $A(K)$ of $K$-points has an infinite center. Then there exist a finite-dimensional $K$-algebra $B$ and a homomorphism of algebraic rings $\psi \colon \tilde{B} \to A$ that induces an isomorphism $\psi_K \colon B = \tilde{B}(K) \to  A(K)$ of abstract rings.
\end{thm}

\noindent The proof of this result is based on the analysis of so-called {\it perfect algebraic rings} and will be given in \S \ref{S-3}.

The rest of the paper is devoted to an application of Theorem \ref{T-1} to rigidity questions for representations of Chevalley groups over rings. Before formulating our second result, we briefly recall the set-up developed in \cite{IR} for the analysis of finite-dimensional representations of Chevalley groups. Suppose $\Phi$ is a reduced irreducible root system of rank $\geq 2$ and let $G = G(\Phi)$ be the corresponding universal Chevalley-Demazure group scheme over $\Z.$ Given a commutative ring $R$, we say that $(\Phi, R)$ is a {\it nice pair} if $2 \in R^{\times}$ if $\Phi$ contains a subsystem of type $B_2$, and $2,3 \in R^{\times}$ if $\Phi$ is of type $G_2.$ Furthermore, we denote by $E(\Phi, R)$ the {\it elementary subgroup} of $G(R)$, i.e. the subgroup generated by the $R$-points of the canonical one-parameter root subgroups. In \cite{IR}, we analyzed abstract finite-dimensional representations
$$
\rho \colon E(\Phi, R) \to GL_n (K)
$$
when $K$ is an algebraically closed field of characteristic 0 and proved that in many situations, such representations admit a {\it standard description}, i.e., after possibly passing to a finite-index subgroup, $\rho$ can be factored as a group homomorphism arising from a ring homomorphism, followed by a morphism of algebraic groups --- see \cite[Main Theorem]{IR} for the precise formulation. In particular, this confirmed a conjecture of Borel and Tits \cite{BT} for split groups.

One of the principal ingredients of the proof was the construction of an algebraic ring $A$ over $K$ naturally associated to the representation $\rho$ --- see \cite[Theorem 3.1]{IR} for the full details as well as \S \ref{S-4} below for a brief reminder. While this part can be carried out over any algebraically closed field, the remainder of the argument in {\it loc. cit} exploited the equivalence of categories given by $\mathcal{F}$ discussed earlier. The structural result of Theorem \ref{T-1} allows one to extend this approach to positive characteristic. In the following statement, we use the notations $G = G(\Phi)$ and $E(\Phi, R)$ introduced above.

\begin{thm}\label{T-2}
Let $K$ be an algebraically closed field of characteristic $p > 0$, $\Phi$ a reduced irreducible root system of rank $\geq 2$, and $R$ a commutative semilocal noetherian ring satisfying $pR = 0$ such that $(\Phi, R)$ is a nice pair. %Denote by $G$ the universal Chevalley-Demazure group scheme of type $\Phi.$
Suppose $\rho \colon E(\Phi, R) \to GL_n (K)$ is a representation and set $H = \overline{\rho (E(\Phi, R))}$ (Zariski closure). Then there exists a commutative finite-dimensional $K$-algebra $B$, a ring homomorphism $f \colon R \to B$ with Zariski-dense image, and a morphism of algebraic groups $\sigma \colon G(B) \to H$ such that for a suitable finite-index subgroup $\Gamma \subset E(\Phi, R)$, we have
$$
\rho \vert_{\Gamma} = (\sigma \circ F) \vert_{\Gamma},
$$
where $F \colon E(\Phi, R) \to G(B)$ is the group homomorphism induced by $f.$
\end{thm}

\noindent (Here, we view the group $G(B)$ as an algebraic group over $K$ using the functor of restriction of scalars.)

%This result will be proved in \S \ref{S-4}.

One noteworthy application of this result, which will be discussed in Proposition \ref{P-Seitz}, is that it allows us to
recover a rigidity theorem of Seitz \cite{Sei}, which was originally proved using techniques similar to those of Borel and Tits \cite{BT}.

The paper is organized as follows. In \S\ref{S-2} we discuss perfect algebraic rings and establish several results that are then used in \S\ref{S-3} to complete the proof of Theorem \ref{T-1}. We then prove Theorem \ref{T-2} in \S\ref{S-4} and discuss some applications.

\vskip2mm

\noindent {\bf Acknowledgments.} The idea to pursue this project originated after the second-named author gave a talk based on \cite{IR} at the University of Michigan. He would like to thank Mitya Boyarchenko and Gopal Prasad for their hospitality. I.R. was supported by an NSF Postdoctoral Fellowship at Harvard University during the preparation of this paper.

\section{Preliminaries on perfect algebraic rings}\label{S-2}

In this section, we review several facts about perfect algebraic rings that will be needed for the proof of Theorem \ref{T-1}. Throughout this section, $K$ will denote a fixed algebraically closed field of characteristic $p > 0.$

Recall that for a scheme $X \in Sch/K$, the {\it absolute Frobenius} is the morphism of schemes $$\Phi \colon X \to X$$ which is the identity on the underlying topological space and the map $f \mapsto f^p$ on the sections of the structure sheaf $\mathcal{O}_X$ (note that $\Phi$ is \emph{not} a morphism of $K$-schemes). In analogy with the case of fields, one says that $X$ is {\it perfect} if $\Phi$ is an isomorphism of schemes. For example, if $R$ is a commutative $K$-algebra, then the affine scheme $\text{Spec}(R)$ is perfect if and only if $R$ is a perfect ring (i.e., the Frobenius map $\varphi \colon R \to R$, $a \mapsto a^p$, is bijective). We refer the reader to \cite{G1}, \cite{S}, and \cite[Appendix A]{Boyar} for further details on perfect schemes.

Let us denote by $Perf/K$ the full subcategory of $Sch/K$ consisting of perfect schemes. In \cite{G1}, Greenberg showed that the inclusion functor $Perf/K \hookrightarrow Sch/K$ admits a left adjoint
$$
Sch/K \to Perf/K, \ \ \ X \mapsto X^{perf},
$$
which, following \cite{Boyar}, we will refer to as the {\it perfectization functor} (Greenberg used the term {\it perfect closure} for $X^{perf}$). For the sake of completeness, as well as to introduce notations that will be needed later, we briefly recall the main points of
the construction, which proceeds in two steps: first one defines the perfectization functor for commutative $K$-algebras and then one ``globalizes" it to the case of $K$-schemes.
More precisely, if $R$ is a commutative $K$-algebra,
then its perfectization $R^{perf}$ is defined as the direct limit of the sequence
$$
R \stackrel{\varphi}{\longrightarrow} R \stackrel{\varphi}{\longrightarrow} R \stackrel{\varphi}{\longrightarrow} \dots,
$$
where $\varphi(x) = x^p.$ One easily shows that
$R^{perf}$ is a {\it perfect} ring equipped with a $K$-algebra structure coming from the structure homomorphism of $K$ into the first term of the above sequence. Moreover, there is a canonical ring homomorphism
$\sigma \colon R \to R^{perf}$, and the pair
$(R^{perf}, \sigma)$ is determined uniquely up to (unique) isomorphism by the following universal property. {\it For any perfect ring $B$ and any ring homomorphism $\psi \colon R \to B$, there exists a unique homomorphism $\psi^* \colon R^{perf}~\to~B$ satisfying $\psi^* \circ \sigma = \psi.$}

%suppose $\psi \colon R \to B$ is a ring homomorphism of $R$ into a perfect ring $B$. Then there exists a unique homomorphism $\psi^* \colon R^{perf}~\to~B$ satisfying $\psi^* \circ \sigma = \psi.$

Next, given $X \in Sch/K$ with structure sheaf $\I_X,$ one shows that there exists a perfect scheme $X^{perf}$ whose underlying topological space is $X$ and whose structure sheaf is given by
$$
U \mapsto \I_X(U)^{perf} \ \ \ \text{for any open} \ U \subset X.
$$
The scheme $X^{perf}$ can be described more explicitly using the absolute Frobenius as follows.
If $\pi \colon X \to \text{Spec}~K$ is the structure morphism, then $\pi \circ \Phi = \Phi \circ \pi$ is another morphism from $X$ to $\text{Spec}~K$, which defines an object of $Sch/K$ that we will denote by $X^{(1/p)}.$ By construction, it is equipped with a natural morhpism $X^{(1/p)} \to X$ of $K$-schemes given by $\Phi.$ Applying the same construction to $X^{(1/p)},$ we obtain another $K$-scheme, denoted by $X^{(1/p^2)}$, together with a $K$-morphism $X^{(1/p^2)} \to X^{(1/p)}.$ Then, in view of the above description of the perfectization of a $K$-algebra, we see that $X^{perf}$ is the inverse limit of the sequence
$$
\cdots \to X^{(1/p^2)} \to X^{(1/p)} \to X.
$$
The natural map $\tau \colon X^{perf} \to X$ is a homeomorphism on the underlying topological spaces, and it induces a bijection $X^{perf} (R) \stackrel{\simeq}{\longrightarrow} X(R)$ for any perfect $K$-algebra $R.$ Moreover, as above, the
pair $(X^{perf}, \tau)$ is universal with respect to morphisms of perfect schemes into $X$. In particular, any morphism $g \colon X \to Y$ of schemes induces a morphism $g^{perf} \colon X^{perf} \to Y^{perf}$ between the corresponding perfectizations.

\vskip2mm

\noindent {\bf Remark 2.1.} Suppose $X, Y \in Var/K$ are affine $K$-varieties and let
$f \colon X^{perf}~\to~Y^{perf}$ be a $K$-morphism. Then, an argument similar to the one used in the proof of \cite[Proposition 2]{S} shows
%imitating the argument in the proof of (\cite{S}, Proposition 2), we see
that there exists a $K$-morphism $X^{1/p^n} \stackrel{g}{\longrightarrow} Y$ such that the composition
$$
X^{(1/p^n), perf} \to X^{perf} \stackrel{f}{\longrightarrow} Y^{perf}
$$
is equal to $g^{perf}$ (where the first arrow is induced by the $n$-fold composition $\Phi^n \colon X^{1/p^n} \to X$).

\vskip2mm

\addtocounter{thm}{1}

Next, notice that
%We will also need to consider perfectizations of algebraic groups and algebraic rings. First, notice that
the perfectization functor preserves products since it has a left adjoint. In particular,
the perfectization of a group or ring object in $Sch/K$ possesses the same structure. We define a {\it perfect algebraic ring} to be a ring object in $Perf/K$ that is isomorphic to $A^{perf}$ for some algebraic ring $A$ over $K$. Similarly, a {\it perfect unipotent group} is a group object in $Perf/K$ that is isomorphic to $G^{perf}$ for some unipotent affine algebraic group $G$ over $K$. Serre \cite{S} showed that the category of commutative perfect unipotent groups over $K$ is an abelian category.
%We will need to consider the category of commutative perfect unipotent groups over $K$, which was shown by Serre \cite{S} to be an abelian category.
The following lemma describes a useful property of morphisms in this category.

\begin{lemma}\label{L-Morph}
A morphism $f \colon G \to H$ of commutative perfect unipotent groups over $K$ is an isomorphism (resp. a monomorphism, an epimorphism) if and only if the induced map on $K$-rational points $f_K \colon G(K) \to H(K)$ is bijective (resp. injective, surjective).
\end{lemma}
\begin{proof}
Though this result is well-known, we were unable to find a direct reference, so we sketch the argument for completeness.
Suppose first that $f_K \colon G(K) \to H(K)$ is injective and let $C$ be the kernel of $f$. We need to show that $C = 0.$ Our assumptions imply that $C$ is zero-dimensional and that $C(K)$ has only one point.
Therefore, our claim is
reduced to the following statement: if $A$ is a perfect commutative finite-dimensional algebra over $K$ such that there is only one $K$-algebra homomorphism $A \to K$, then $A \simeq K.$ Indeed, since $A$ is perfect, it is reduced,
%has no nilpotents elements,
so its Jacobson radical is zero. Since $K$ is algebraically closed, it follows that $A \simeq K \times \cdots \times K$; but, by our assumption, there is a unique morphism $A \to K$, so $A \simeq K,$ as required. In the case that $f_K$ is surjective, one uses the same argument with $C = \text{coker}~f.$
%we take $C$ to be the cokernel of $f$ and repeat the same argument.
\end{proof}

%For example, if $G$ is a unipotent group over $K$, then $G^{perf}$ is a perfect group scheme over $K$. We will use the term {\it perfect unipotent group} to refer to a group object in $Perf/K$ that is isomorphic to $G^{perf}$ for some unipotent group over $K$. The category of commutative perfect unipotent groups over $K$ was studied by Serre \cite{S}. It is an abelian category, and a morphism $f \colon G \to H$ in this category is an isomorphism (resp. a monomorphism, resp. an epimorphism) if and only if the induced map on $K$-rational points $G(K) \to H(K)$ is bijective (resp. injective, resp. surjective) ({\bf need reference for this fact}). Similarly, a {\it perfect algebraic ring} over $K$ is a ring object in $Perf/K$ that is isomorphic to $A^{perf}$ for some algebraic ring $A$ over $K.$

\section{Proof of Theorem \ref{T-1}}\label{S-3}

We now turn to the proof of Theorem \ref{T-1}. As before, $K$ will denote an algebraically closed field of characteristic $p > 0.$

The key step in the argument is to show that perfect algebraic rings of characteristic $p$ arise from finite-dimensional $K$-algebras.
%For the argument, we will first pass to perfect algebraic rings of characteristic $p$ and show that such rings basically arise from finite-dimensional $K$-algebras.
The precise statement is as follows.

\begin{prop}\label{P-PerfRing}
Let $A$ be a connected perfect algebraic ring over $K$. Assume that $p A = 0$ and the ring $A(K)$ of $K$-points has an infinite center.
%such that $pA = 0$ and the center of the ring $A(K)$ of $K$-points is infinite.
Then there exists a finite-dimensional $K$-algebra $B$ such that $A \simeq \tilde{B}^{perf}$ as perfect algebraic rings over $K$.
\end{prop}

We will deduce this statement from the next two lemmas. The first gives
%We will deduce Proposition \ref{P-PerfRing} from the two lemmas that follow. The first of these gives
some information on the structure of {\it commutative} perfect algebraic rings of characteristic $p.$

\begin{lemma}\label{L-PerfRing1}
Let $C$ be a nontrivial connected commutative perfect algebraic ring over $K$ such that $p C = 0$.  Then there exists a morphism $\tilde{K}^{perf} \to C$ of perfect algebraic rings over $K$.
\end{lemma}
\begin{proof}
First, without loss of generality, we may assume that $C$ is indecomposable, since if the statement is true for rings $C_1$ and $C_2$, then it is also true for their product. Using \cite[Proposition 8.1]{G} or \cite[Proposition 2.20]{IR}, we can then write
$C = C_0^{perf},$ where $C_0$ is a local connected algebraic ring. Consequently, there exists a surjective morphism $f_0 \colon C_0 \to \tilde{K}$ whose kernel $J_0$ is a nilpotent ideal of $C_0$ (see \cite[Proposition 7.1]{G} or \cite[Proposition 2.19]{IR}). Now, since the underlying additive group of an algebraic ring is unipotent (see \cite[Proposition 4.3]{G}), Lemma \ref{L-Morph} implies that on passing to perfectizations, we obtain
%by \cite[Proposition 7.1]{G} and \cite{IR}, Proposition 2.19, there exists a surjective morphism $f_0 \colon C_0 \to \tilde{K}$ whose kernel $J_0$ is a nilpotent ideal of $C_0$. Passing to perfectizations and using the fact that the underlying additive group of an algebraic ring is unipotent (see \cite{G}, Proposition 4.3), this implies that we have
a surjective morphism $f = f_0^{perf} \colon C \to \tilde{K}^{perf}$ whose kernel $J = J_0^{perf}$ is a nilpotent ideal of $C$. Take $n$ sufficiently large so that $J^{p^n} = 0.$ Then the map $C \to C$, $x \mapsto x^{p^n}$,
%In particular, we can find $n$ sufficiently large so that $J^{p^n} = 0.$ Since $C$ is commutative and $p C = 0$, the map $C \to C$ given by $x \mapsto x^{p^n}$
is a morphism of perfect algebraic rings,  which, by construction, factors through $\tilde{K}^{perf} \to C,$ yielding the required morphism.
\end{proof}

For the second lemma, we recall that a perfect unipotent group $V$ over $K$ is said to be a $\tilde{K}^{perf}$-module if there exists a morphism $\tilde{K}^{perf} \times V \to V$ satisfying the usual module axioms.

\begin{lemma}\label{L-PerfRing2}
Suppose $V$ is a perfect unipotent group over $K$ equipped with the structure
of a $\tilde{K}^{perf}$-module. %i.e. there exists a morphism $\tilde{K}^{perf} \times V \to V$ satisfying the usual module axioms.
Then there exists an isomorphism $\G_a^d \stackrel{\simeq}{\longrightarrow} V$ of perfect unipotent groups that intertwines the natural action of $\tilde{K}^{perf}$ on $\G_a^d$ by scaling and the given action of $\tilde{K}^{perf}$ on $V$.
\end{lemma}
\begin{proof}
By construction, we have $\tilde{K}^{perf}(K) = K$, so $V(K)$ is naturally a $K$-vector space. Let $v_1, \dots, v_n \in V(K)$ be any finite linearly independent subset. Using the $\tilde{K}^{perf}$-action on $V$, we obtain a morphism $\G_a^n \stackrel{f}{\longrightarrow} V$ of $\tilde{K}^{perf}$-modules, where the action of $\tilde{K}^{perf}$ on $\G_a^n$ is given simply by scaling. Moreover, the map $K^n \to V(K)$ induced by $f$ is injective, so $f$ is a monomorphism by Lemma \ref{L-Morph}, and hence $n \leq d := \dim V.$ Thus, $V(K)$ must be finite-dimensional over $K$. Furthermore, if in the above construction the set $v_1, \dots, v_n$ is a $K$-basis of $V(K)$, then by the same argument, $f$ is an isomorphism of perfect unipotent groups over $K$, which completes the proof.
\end{proof}

\vskip2mm

\noindent {\it Proof of Proposition \ref{P-PerfRing}.} Let $C$ be the connected component of the additive identity of the center of $A$.
%Let $A$ be as in the statement of the proposition, and let $C$ be the connected component of the center of $A$.
By our assumption, $C$ is nontrivial, so by Lemma \ref{L-PerfRing1}, there exists a morphism $\tilde{K}^{perf} \to C$ of perfect algebraic rings over $K$, which gives $A$ the structure of a $\tilde{K}^{perf}$-algebra object in the category $Perf/K.$ By %\cite[Ch. VII, Proposition 11]{S1} and
Lemma \ref{L-PerfRing2}, we may assume that the underlying additive group of $A$ comes from a finite-dimensional $K$-vector space. Since the multiplication in $A$ must be given by a $K$-bilinear map, we conclude that there exists a finite-dimensional $K$-algebra $B$ such that $A \simeq \tilde{B}^{perf}$, as needed. \hfill $\Box$

%Then the multiplication in $A$ must be given by a $K$-bilinear map, so we automatically obtain a finite dimensional $K$-algebra. \hfill $\Box$

\vskip2mm

We can now complete

\vskip2mm

\noindent {\it Proof of Theorem \ref{T-1}.} Suppose $A$ is a connected algebraic ring over $K$ such that
$pA = 0$ and the ring $A(K)$ has infinite center. Then, according to Proposition \ref{P-PerfRing}, we can find a finite-dimensional $K$-algebra $B$ and an isomorphism $\tilde{B}^{perf} \stackrel{\simeq}{\longrightarrow} A^{perf}$ of perfect algebraic rings over $K$. By Remark 2.1, this isomorphism comes from a morphism $\tilde{B}^{(1/p^n)} \stackrel{g}{\longrightarrow} A$ of algebraic rings over $K$ for some sufficiently large $n.$ It follows from our construction and Lemma \ref{L-Morph} that $g$ induces an isomorphism on the rings of $K$-points. It remains to check that $\tilde{B}^{(1/p^n)}$ also comes from a finite-dimensional $K$-algebra. For this, pick a basis of $B$ as a $K$-vector space and use it to identify $\tilde{B}$ with $\G_a^d$ as an additive group for some $d \in \N.$ The algebra structure on $B$ is then defined by means of certain structure constants. We can also identify $\tilde{B}^{(1/p^n)}$ with $\G^d_a$ in a natural way, so that the canonical map $\tilde{B}^{(1/p^n)} \to \tilde{B}$ gets identified with the map $\G_a^d \to \G_a^d$ that raises all coordinates to the power $p^n$. Then, taking the
the $p^n$-th roots of the structure constants for $B$, we obtain structure constants for an algebra that defines $\tilde{B}^{(1/p^n)}.$
\hfill $\Box$

\section{Proof of Theorem \ref{T-2} and applications}\label{S-4}

In this section, we will apply Theorem \ref{T-1} to rigidity questions and, in particular, use it to establish Theorem \ref{T-2}. As a consequence,
in Proposition \ref{P-Seitz}, we will give a new proof of a rigidity statement due originally to Seitz \cite{Sei}.

%how considerations involving algebraic rings and derivations enable one to give a new proof of a rigidity statement that was originally proved by Seitz \cite{Sei} using the framework of Borel-Tits \cite{BT}.

We begin with the following proposition, which is an adaptation of \cite[Theorem 3.1]{IR} to our current setting.

\begin{prop}\label{P-1}
Let $K$ be an algebraically closed field of characteristic $p > 0,$ $\Phi$ a reduced irreducible root system of rank $\geq 2$, and $R$ a commutative noetherian ring satisfying $p R = 0$ such that $(\Phi, R)$ is a nice pair. Denote by $G$ the universal Chevalley-Demazure group scheme over $\Z$ of type $\Phi$ and consider a linear representation of the elementary subgroup $\rho \colon E(\Phi, R) \to GL_n (K)$. Then there exists a finite-dimensional commutative $K$-algebra $B$ and a finite ring $C$, together with a ring homomorphism $\tilde{f} \colon R \to B':= B \oplus C$ having Zariski-dense image such that for each root $\alpha \in \Phi$, there is an injective regular map $\varphi_{\alpha} \colon B' \to H$ into $H:= \overline{\rho (E(\Phi, R))}$ satisfying
\begin{equation}\label{E-1}
\rho (e_{\alpha} (t)) = \varphi_{\alpha} (\tilde{f}(t))
\end{equation}
for all $t \in R.$
%\footnotetext{Recall that this means that $2 \in R^{\times}$ if $\Phi$ contains a subsystem of type $B_2$, and $2,3 \in R^{\times}$ if $\Phi$ is of type $G_2.$}
\end{prop}
\begin{proof}
First, since $(\Phi, R)$ is assumed to be a nice pair, it follows from \cite[Theorem 3.1]{IR} that there exists a commutative algebraic ring $A$, together with a ring homomorphism $f \colon R \to A(K)$ having Zariski-dense image such that
for every root $\alpha \in \Phi$, there is an injective regular map $\psi_{\alpha} \colon A(K) \to H$ satisfying
$$
\rho (e_{\alpha} (t)) = \psi_{\alpha} (f(t))
$$
for all $t \in R.$ Next, our assumption that $R$ is noetherian implies that we can write
$$
A(K) = A' (K) \oplus C,
$$
where $A'$ is a connected algebraic ring and $C$ is a finite ring (see \cite[Proposition 2.11 and Lemma 2.13]{IR}).
Let $\pi \colon A(K) \to A'(K)$ be the natural projection and set $f' = \pi \circ f.$
Since $R$ has characteristic $p$, it follows that $p A' = 0.$ Therefore, by Theorem \ref{T-1}, there exists a finite-dimensional commutative $K$-algebra $B$ and a morphism of algebraic rings $\varphi \colon \tilde{B} \to A'$ that induces a ring isomorphism $\varphi_K \colon B= \tilde{B} (K) \to A'(K)$ on the rings of $K$-points.
Let $\psi_K \colon A'(K) \to B$ be the inverse homomorphism (of abstract rings) and define $\tilde{f}_0 \colon R \to B$ by $\tilde{f}_0 = \psi_K \circ f'.$

By dimension considerations, it is easy to see that $\tilde{f}_0 (R)$ is Zariski-dense in $B$. Now set $B' := B \oplus C$ and define $\varphi_K' \colon B' \to A(K)$ by $\varphi_K' \vert_B = \varphi_K$ and $\varphi_K' \vert_C = \text{id}_C.$
Also, for any root $\alpha \in \Phi,$ let $\varphi_{\alpha} \colon B' \to H$ be the composition $\varphi_{\alpha} = \psi_{\alpha} \circ \varphi_K'.$ Finally, let $\pi_C \colon A(K) \to C$ be the natural projection and define $\tilde{f} \colon R \to B'$ by $\tilde{f} = \tilde{f}_0 \oplus (\pi_C \circ f).$ Then %above discussion then shows that
$\tilde{f}$ has Zariski-dense image, and from the construction, it is clear that condition (\ref{E-1}) is satisfied.
\end{proof}

Using Proposition \ref{P-1}, the proof of Theorem \ref{T-2} can now be carried out following the strategy developed in \cite{IR}, to which we refer the reader for full details. We sketch the main points of the argument here for completeness.

\vskip2mm

\noindent {\it Proof of Theorem \ref{T-2}.} Let $K$, $\Phi$, and $R$ be as in the statement of the theorem, and consider a representation
$$
\rho \colon E(\Phi, R) \to GL_n (K).
$$
By Proposition \ref{P-1}, we can associate to $\rho$ a ring homomorphism $\tilde{f} \colon R \to B' = B \oplus C$ with Zariski-dense image, where $B$ is a finite-dimensional commutative $K$-algebra, together with regular maps $\varphi_{\alpha} \colon B' \to H$ (for $\alpha \in \Phi$) satisfying (\ref{E-1}). Using this, as well as \cite[Proposition 4.2]{IR}, we lift $\rho$ to a representation $\tilde{\tau} \colon \text{St}(\Phi, B') \to GL_n (K)$ of the corresponding Steinberg group. The fact that $R$ is semilocal then implies that $\tilde{\sigma} : = \tilde{\tau} \vert_{\text{St} (\Phi, B)}$ descends to a representation $\sigma \colon G(B) \to GL_n (K)$ such that
$$
\rho \vert_{\Gamma} = (\sigma \circ F) \vert_{\Gamma}
$$
on a suitable finite-index subgroup $\Gamma \subset E(\Phi, R)$, where $F$ is the group homomorphism induced by $\pi \circ \tilde{f}$ with $\pi \colon B' \to B$ being the natural projection. Finally, we view $G(B)$ as an algebraic group over $K$ using the functor of restriction of scalars and conclude from \cite[Proposition 6.3]{IR} that $\sigma$ is in fact a morphism of algebraic groups over $K$. \hfill $\Box$

\vskip2mm

We note that other results from \cite{IR} can also be developed in the present setting in a similar fashion.

\vskip5mm

\noindent {\bf Remark 4.2.} We should point out since results in the spirit of the conjecture of Borel-Tits \cite{BT} deal with algebras, the use of Theorem \ref{T-1} in positive characteristic cannot in general be avoided. Namely, let us again consider
the algebraic ring $A$ whose underlying additive group is $\G_a \oplus \G_a$ and whose multiplication operation is given by
$$
(x_1, y_1) \cdot (x_2, y_2) = (x_1 x_2, x_1^p y_2 + x_2^p y_1).
$$
Notice that the map
$$
(x,y) \mapsto \left( \begin{array}{ccc} x^p & 0 & y \\ 0 & x & 0 \\ 0 & 0 & x^p \end{array} \right).
$$
defines an embedding ${A} (K) \stackrel{\iota}{\hookrightarrow} M_3 (K).$
Let now $B = K[\varepsilon]$ be the $K$-algebra of dual numbers and $\tilde{B}$ be the corresponding algebraic ring. Recall that we have a morphism
$$
\varphi \colon \tilde{B} \to A \ \ \ \ (x,y) \mapsto (x, y^p)
$$
that induces an isomorphism $\varphi_K \colon B \to A(K)$ on the rings of $K$-points.

Next, set $R = K[X]$ and consider the surjective homomorphism
$$
\chi \colon R \to B, \ \ \ g(X) \mapsto g(0) + g'(0) \varepsilon,
$$
where $g'(X)$ is the (formal) derivative of $g(X).$ Then the ring homomorphism $$f = \varphi_K \circ \chi \colon R \to A(K)$$ yields, via the embedding $\iota$, a representation
$$
\rho \colon SL_3 (R) \to GL_9 (K).
$$
Since the algebraic ring $A_{\rho}$ associated to $\rho$ is constructed as the Zariski-closure of the image of a root subgroup (see \cite[Theorem 3.1]{IR}), it follows that $A_{\rho} = A$. In particular, $A_{\rho}$ does not arise from a $K$-algebra.

\vskip5mm

To conclude this section, we will now give an alternative proof (for Chevalley groups of rank $\geq 2$) of a result of
Seitz \cite{Sei} on abstract homomorphisms of the groups of rational points of algebraic groups.
%into the framework of algebraic rings. This differs from Seitz's original approach, which relied on the techniques developed by Borel and Tits in \cite{BT}.
%of the paper \cite{BT} by Borel and Tits.
%The set-up is as follows.
Let $k$ be an infinite perfect field of characteristic $p > 0$ and let $K$ be an algebraically closed field of the same characteristic. Denote by $G$ the universal Chevalley group associated to a reduced irreducible root system $\Phi,$ and suppose that $Y$
is an algebraic group over $K.$ Seitz showed that any (nontrivial) abstract homomorphism $G(k) \to Y(K)$ can be factored as
\begin{equation}\label{E-Seitz}
G(k) \stackrel{\varphi}{\longrightarrow} G(K) \times \cdots \times G(K) \stackrel{\sigma}{\longrightarrow} Y(K),
\end{equation}
where $\sigma$ is a morphism of algebraic groups and each component $G(k) \to G(K)$ of $\varphi$
arises from a field homomorphism $k \to K$ (see \cite[Theorem 1]{Sei}). On the other hand, our Theorem \ref{T-2} gives a description of abstract homomorphisms over arbitrary (i.e. also imperfect) fields of characteristic $p > 0$, and enables us to recover the result of Seitz (see Proposition \ref{P-Seitz} below) by observing that a \emph{perfect} field of positive characteristic does not have nontrivial derivations. This approach also explains why Seitz's result is no longer true over imperfect fields: the nontrivial derivations can be used to produce abstract homomorphisms of a different nature from the ones in (\ref{E-Seitz}) --- see the construction of Borel and Tits in \cite[8.18]{BT}.
%The argument given below explains why the assumption that $k$ is perfect is necessary, and in fact shows that the result is no longer true over imperfect fields. Indeed, an imperfect field admits nontrivial derivations, which enables one to use the construction of Borel-Tits \cite[8.18]{BT} to produce abstract homomorphisms that are not of the form (\ref{E-Seitz}). In contrast to Seitz's original approach, Theorem \ref{T-2} allows one to consider arbitrary fields; specializing to the case of perfect fields, we obtain the following.
In the statement below, we are implicitly using the well-known equality $E(\Phi, k) = G(k)$.

\addtocounter{thm}{1}

\begin{prop}\label{P-Seitz}
Let $k$ be an infinite perfect field of characteristic $p > 0$, $K$ an algebraically closed field of the same characteristic, and $\Phi$ a reduced irreducible root system of rank $\geq 2$ such that $(\Phi, k)$ is a nice pair.
Let $G$ be the universal Chevalley-Demazure group scheme of type $\Phi,$ and consider a representation $\rho \colon G(k) \to GL_n (K).$ Then there exist

\vskip2mm

\noindent {\rm (i)} a finite-dimensional $K$-algebra
$$
B \simeq K^{(1)} \times \cdots \times K^{(r)}
$$
with $K^{(i)} \simeq K$ for all $i$;

\vskip1mm

\noindent {\rm (ii)} an embedding $f \colon k \hookrightarrow B$ with Zariski-dense image; and

\vskip1mm

\noindent {\rm (iii)} a morphism of algebraic groups $\sigma \colon G(B) \to GL_n (K)$

\vskip1mm

\noindent such that
$$
\rho = \sigma \circ F,
$$
where $F \colon G(k) \to G(B)$ is the group homomorphism induced by $f.$

\end{prop}

Before turning to the proof, we need to introduce the following notation. Let
$R$ be a commutative ring, $K$ a field, and $g \colon R \to K$ a ring homomorphism. We will denote by
$\mathrm{Der}^g (R,K)$ the space of $K$-valued derivations of $R$ with respect to $g$, i.e. an element $\delta \in \mathrm{Der}_g (R,K)$ is a map $\delta \colon R \to K$ such that for any $r_1, r_2, \in R$,
$$
\delta (r_1 + r_2) = \delta (r_1) + \delta(r_2) \ \ \ \mathrm{and} \ \ \  \delta (r_1 r_2) = \delta(r_1) g(r_2) + \delta (r_2) g(r_1).
$$

\vskip2mm

\noindent {\it Proof of Proposition 4.3.}
By Theorem \ref{T-2}, there exists a finite-dimensional commutative $K$-algebra
$B$ and an embedding $f \colon k \hookrightarrow B$ with Zariski-dense image, together with the required morphism of algebraic groups $\sigma \colon G(B) \to GL_n (K).$ Moreover, since $k$ is an infinite field, it is well-known that $E(\Phi, k) = G(k)$ contains no proper noncentral normal subgroups (see \cite{T1}), so $\Gamma = E(\Phi, k)$.
Thus, it remains to show that $B$ is isomorphic to a product of copies of $K$.

For this, let $J$ be the Jacobson radical of $B$. Since $K$ is perfect, by the Wedderburn-Malcev Theorem, we can find a semisimple subalgebra $B' \subset B$ such that $B = B' \oplus J$ as $K$-vector spaces and
$$
B' \simeq B/J \simeq \underbrace{K \times \cdots \times K}_{r \ \text{copies}}
$$
as $K$-algebras (cf. \cite[Corollary 11.6]{P}). Let $e_i \in B'$ be the $i$th standard basis vector. Since $e_1, \dots, e_r$ are idempotent, and we have $e_1 + \cdots + e_r = 1$ and $e_i e_j = 0$ for $i \neq j,$ it follows that we can write
$B = \oplus_{i=1}^{r} B_i,$ where $B_i = e_i B.$ Clearly, $B_i = B'_i \oplus J_i$, with $B'_i = e_i B' \simeq K$ and $J_i = e_i J$; in particular, $B_i$ is a local $K$-algebra with maximal ideal $J_i$. To complete the proof, it suffices to show that $J_i = \{ 0 \}$ for all $i,$ so, we may now assume that $B$ is a commutative finite-dimensional local algebra of the form $B = K \oplus J.$ Let $B'' = B/J^2,$ and fix a $K$-basis $\{\varepsilon_1, \dots, \varepsilon_s \}$ of $J/J^2.$ Denote by $g \colon k \to B''$ the composition of $f$ with the canonical map $B \to B''$
%(note that $f'$ has Zariski-dense image).
Then for any $x \in k,$ we have
$$
g (x) = f_0(x) + \delta_1 (x) \varepsilon_1 + \cdots + \delta_s (x) \varepsilon_s,
$$
where $f_0 \colon k \to K$ is a field homomorphism and $\delta_1, \dots, \delta_s \in \mathrm{Der}^{f_0} (k,K).$ But since $k$ is perfect (i.e. any element is a $p^{\text{th}}$ power), all $K$-valued derivations on $k$ are trivial, so $J/J^2 = 0$ as $g$ has Zariski-dense image. Hence $J = J^2,$ and consequently $J = \{ 0 \}$ by Nakayama's Lemma. This completes the proof. $\Box$

\vskip5mm

\bibliographystyle{amsplain}

\end{document}